\documentclass{amsart}
\usepackage{amsfonts}
\usepackage{graphicx}
\usepackage{amscd}
\usepackage{amsmath}
\usepackage{amssymb}
\usepackage{stmaryrd} %produto kulkarni-nomizo
\usepackage{enumitem}

\usepackage{tikz-cd}
\usepackage[pdf]{pstricks}
\usepackage{chemfig}
\usepackage{chemmacros}
\usetikzlibrary{cd}

\makeatletter
\@namedef{subjclassname@2010}{%
	\textup{2010} Mathematics Subject Classification}
\makeatother

\theoremstyle{plain}

\newtheorem{corollary}{\bf Corollary}

\newtheorem{proposition}{\bf Proposition}

\newtheorem{theorem}{\bf Theorem}

\theoremstyle{definition}

\numberwithin{equation}{section}

\title[Hamilton-Ivey estimates for the Ricci-Bourguignon flow with $\rho<0$]{Hamilton-Ivey estimates for the Ricci-Bourguignon flow with $\rho<0$}

\author{Valter Borges}

\address[V. Borges]{Faculdade de Matem\'{a}tica, Universidade Federal do Par\'{a}\\
	66075-110 Be\-l\'{e}m, Par\'{a}, Brazil.}
\email{valterborges@ufpa.br}


\begin{document}

\begin{abstract}
In this paper, we prove Hamilton-Ivey estimates for the Ricci-Bourguignon flow on a compact manifold, with $n=3$ and $\rho<0$. As a consequence, we prove that compact ancient solutions have nonnegative sectional curvature for all negative $\rho$.
\end{abstract}

%\begin{abstract}
%	The goal of this paper is to show that compact ancient solutions of the Ricci-Bourguignon flow have nonnegative sectional curvature for $\rho<0$ and to provide different versions of the Hamilton-Ivey estimate for such values of $\rho$.
%\end{abstract}

%\begin{abstract}
%In this paper we provide Hamilton-Ivey type estimates subject to different positivity curvature assumptions and prove that all ancient solutions with $\rho<0$ have nonnegative.
%\end{abstract}

\date{}

\maketitle{}

\section{Main Results}
%Hamilton-Ivey inequality is... [Ivey], [Hamilton], [Hamilton]
Following \cite{catino1}, we say that a 1-parameter family of Riemannian metrics $\{g(t)\}_{t\in[0,T)}$ on a manifold $M^{n}$ is a solution of the \textit{Ricci-Bourguignon flow} if it satisfies the equation
\begin{equation}\label{defRB}
	\frac{\partial}{\partial t}g=-2(Ric-\rho Rg),
\end{equation}
where $\rho$ is a real parameter. The tensor on the right hand side of (\ref{defRB}) has special interest for certain values of $\rho$. For instance, if $\rho=0$ it is the Ricci tensor and the corresponding flow gives rise to the Ricci flow. When $\rho=1/2(n-1)$ the corresponding tensor is the Schouten tensor, an important tensor in conformal geometry, which gives rise to the Schouten flow. 

The Ricci flow is an important tool in the investigation of geometry and topology of three manifolds. In \cite{ham3}, it was shown that the universal covering of a compact three manifold with positive Ricci curvature is isometric to $\mathbb{S}^{3}$. One important feature here is that the Ricci flow develops singularities everywhere at the same finite time. When the initial metric does not have positive Ricci curvature, the situation changes completely, once the manifold may develop complicated singularities \cite{hamilton3}. Despite this complex behavior, Hamilton \cite{ham1,hamilton3} and Ivey \cite{ivey}, independently, discovered a property satisfied by any solution of the Ricci flow, known as Hamilton-Ivey estimate, which says that under the Ricci flow the metric pinches towards positive curvature. This estimate was extended to $\rho\in[0,1/6)$, and to $\rho\in[1/6,1/4)$ under the assumption that the scalar curvature is nonnegative (however, this hypothesis can be removed, see Theorem \ref{HYestimate}). On the other hand, it is stated in \cite[Chapter $5$]{cremaschi} that a Hamilton-Ivey estimate similar to \cite[Theorem 4.1]{ham1} is not expected to be true when $\rho<0$.

Motivated by results of Ivey, \cite[Theorem 2]{ivey}, and Hamilton, \cite[Theorem 24.4]{hamilton3}, for $\rho=0$, we prove the following estimate, which helps to understand the behavior of the Ricci-Bourguignon flow for $\rho<0$. Consider the convex increasing function $f(x)$, $x\in[e^{1-4\rho},+\infty)$, defined by
\begin{align}\label{deff}
	f(x)=\frac{1}{2(1-2\rho)}x(\log x-2(1-2\rho))\in\left[-\frac{e^{1-4\rho}}{2(1-2\rho)},+\infty\right).
\end{align}

\begin{theorem}\label{HYestimate2}
	Let $M^{3}$ be a compact three manifold, $\rho<0$ and $g_{0}$ a Riemannian metric on $M$. Let $g(t)$, $t\in[0,T)$, be a solution of $(\ref{defRB})$ with $g(0)=g_{0}$, corresponding to $\rho$. If
	\begin{align}\label{estimate2}
		R\geq-\frac{e^{1-4\rho}}{1-2\rho}\ \ \text{and}\ \ Ric\geq-f^{-1}\left(R/2\right)g
	\end{align}
	at $t=0$, then these inequalities continue to be satisfied for $t\in(0,T)$.
\end{theorem}

In what follows we denote the eigenvalues of the curvature operator of a solution of $(\ref{defRB})$ by $\lambda\geq\mu\geq\nu$. Thus, the eigenvalues of its Ricci tensor are $\lambda+\mu\geq\lambda+\nu\geq\mu+\nu$ and its scalar curvature is $R=2(\lambda+\mu+\nu)$.

Any metric $g_{0}$ on a compact $M^3$ can be rescaled to satisfy $(\ref{estimate2})$. This is achieved by asking the least eigenvalue of the Ricci curvature of $g_{0}$ to satisfy $\mu+\nu\geq-\frac{e^{1-4\rho}}{3(1-2\rho)}.$
%\begin{align*}
%	\mu+\nu\geq-\frac{e^{1-4\rho}}{3(1-2\rho)}.
%\end{align*}
In this case, if at a positive time $t_{0}$ the metric $g(t_{0})$ has a Ricci curvature negative enough, say $\mu+\nu\leq-e^{2(1-4\rho)}$, then $(\ref{estimate2})$ becomes equivalent at $t_{0}$ to
\begin{eqnarray*}
	\lambda\geq\frac{1}{2(1-2\rho)}|\mu+\nu|\log|\mu+\nu|,
\end{eqnarray*}
which means that a very negative Ricci curvature only occurs if there is a positive sectional curvature which is much larger than its absolute value.

When the scalar curvature of $g_{0}$ is positive, another interpretation for Theorem $\ref{HYestimate2}$ goes as follows: once $M$ is compact, $g(t)$ develops singularity at a finite time $T$, and thus, $R$ goes to $+\infty$ as we get close to $T$; once $\lim_{y\rightarrow+\infty}f^{-1}(y)/y=0$, $(\ref{estimate2})$ implies that the Ricci curvature tends to becoming nonnegative. It turns out that if the the initial metric has nonnegative scalar curvature and some negative Ricci curvature, then we are able to obtain the following time dependent estimate.
\begin{theorem}\label{HYestimatewithR}
	Let $M^{3}$ be a compact manifold, $\rho<0$ and $g_{0}$ a Riemannian metric on $M$ with nonnegative scalar curvature. If $g(t)$, $t\in[0,T)$, is the solution of the Ricci-Bourguignon flow corresponding to $\rho$, satisfying $g(0)=g_{0}$, then
	\begin{eqnarray}\label{estimate3}
		R\geq\frac{1}{1-2\rho}|\mu+\nu|(\log|\mu+\nu|+\log(1-4\rho t)-2(1-2\rho)),
	\end{eqnarray}
	at any point $(p,t)$ where $\mu+\nu$, the smallest Ricci curvature of $g(t)$, is negative.
\end{theorem}

If, on the other hand, the initial metric has nonnegative Ricci curvature, then we have an estimate similar to \cite[Theorem 4.1]{ham1}.
\begin{theorem}\label{HYestimateRicci>0}
	Let $M^{3}$ be a compact three manifold, $\eta>0$, $\rho\in(-1/\eta,0)$ and $g_{0}$ a Riemannian metric on $M$ with nonnegative Ricci curvature and satisfying $\displaystyle\min_{p\in M}\nu_{0}(p)\geq-1$, where $\nu_{0}$ is the smallest sectional curvature of $g_{0}$. If $g(t)$, $t\in[0,T)$, is the solution of $(\ref{defRB})$ with $g(0)=g_{0}$, corresponding to $\rho$, then the scalar curvature $R(t)$ of $g(t)$ satisfies
	\begin{eqnarray}\label{estimate<0}
		R\geq-\frac{1}{\rho}|\nu|(\log|\nu|+\log(1+2(1+\eta\rho)t)+6\rho),
	\end{eqnarray}
	at any point $(p,t)$ where the smallest sectional curvature  $\nu(p,t)$ of $g_{p}(t)$ is negative.
\end{theorem}

A solution for the Ricci-Bouguignon flow is called \textit{ancient\ solution} if it is defined in an interval of the form $(-\infty,t_{0}]$, for a $t_{0}\in\mathbb{R}$. These solutions may model singularities that form in finite time thus, it is important to understand its geometry. It is known that such solutions have nonnegative scalar curvature, for $\rho\leq1/2(n-1)$. When $n=3$ and $\rho=0$, Hamilton \cite{ham1} showed that ancient solutions have nonnegative sectional curvature. The latter result was extended later by Catino {\it et al.} in \cite{catino1} for $\rho\in[0,1/4)$. These results follow from the respective Hamilton-Ivey estimates and a parabolic dilation, which takes in advantage the fact that ancient solutions are defined for all negative $t$. The details can be found in \cite{RFI}. Following the same strategy of \cite{RFI}, Theorem \ref{HYestimatewithR} and Theorem \ref{HYestimateRicci>0} can be used to extend this result on ancient solutions to all negative values of $\rho$.

\begin{theorem}\label{geometricconsequence}
	Any compact ancient solution of the Ricci-Bourguignon flow corresponding to $\rho<0$ has nonnegative sectional curvature.
\end{theorem}
\begin{proof}
	Let $\rho<0$, fix $\eta>0$ so that $\rho>-1/\eta$ and let $g(t)$ be an ancient solution of $(\ref{defRB})$. As $g(t)$ has nonnegative scalar curvature, we use Theorem \ref{HYestimatewithR} and a parabolic dilation to argue that these metrics cannot have a negative Ricci curvature. Now, using the fact that $g(t)$ has nonnegative Ricci curvature, we use Theorem \ref{HYestimateRicci>0} and a parabolic dilation to show that $g(t)$ cannot have a negative sectional curvature. This finishes the proof of Theorem \ref{geometricconsequence}.
\end{proof}

As a consequence of Theorem \ref{geometricconsequence} and the results mentioned above, we have

\begin{corollary}\label{corollary}
	Any compact ancient solution of the Ricci-Bourguignon flow corresponding to $\rho<1/4$ has nonnegative sectional curvature.
\end{corollary}

This paper is organized in the following way. In the next section we recall important facts concerning the Ricci-Bourguignon flow, stated and proved in \cite{catino1}, in order to prove our results. In the third and last section we prove our estimates. We also give a proof of the Hamilton-Ivey estimate when $\rho\in[1/6,1/4)$ which does not depend on assumptions on the scalar curvature (see Theorem \ref{HYestimate}).

\section{The equation of the curvature tensor}
The curvature tensor $Rm(t)$ of a Ricci-Bourguignon flow $g(t)$ satisfies a reaction-diffusion equation, which possesses a complicated quadratic reaction term. To understand this reaction term, one performs the Uhlembeck's trick, which consists, roughly speaking, of considering a family of linear bundle isometries $\varphi(t):(V,h)\rightarrow(TM,g(t))$, where $(V,h)$ is a metric bundle over $M$, carefully defined so that the pull-back $P(t)=\varphi(t)^{*}Rm(t)$ satisfies a much simpler equation, which we present below. For more details, see \cite{catino1,cremaschi,HRF,RFTAII}.
\begin{proposition}[\cite{catino1}]\label{propP(t)satisfiesPDE}
	The family of operators $P(t):\wedge^{2}V\rightarrow\wedge^{2}V$, $t\in[0,T)$, has the same eigenvalues as $Rm(t)$. Furthermore, $P(t)$ evolves according to the equation
	\begin{equation}\label{PDE}
		\frac{\partial}{\partial t}P=\mathcal{L}P+2P^{2}+2P^{\#}-4\rho tr_{g_{0}}(P)P,
	\end{equation}
	where $\mathcal{L}$ is a differential operator, which is uniformly elliptic if $\rho<1/(2(n-1))$.
\end{proposition}

Once the differential operator above is uniformly elliptic for $\rho<1/(2(n-1))$, it is possible to use the {\it Tensor Maximum Principle} \cite[Theorem 4.8]{catino1} or its time dependent version \cite[Theorem 4.9]{catino1} to investigate $P(t)$ through equation $(\ref{PDE})$.
% In the statement below, $(E,h,D(t))$, $t\in[0,T)$, denotes a vector bundle over $(M,g(t))$, where $h$ is a bundle metric and $D(t)$ is a family of linear connections compatible with $h$.
%
%\begin{theorem}[\cite{catino1}]\label{MPS} Let $u:[0,T)\rightarrow\Gamma(E)$ be a smooth solution of the parabolic equation 
%	\begin{equation}\label{parabequat}
%		\frac{\partial}{\partial t}u=\mathcal{L}u+F(u,t)
%	\end{equation}
%	and $F:E\times[0,T)\rightarrow E$ a continuous map, locally Lipschtiz in the $E$ factor and $F(v,t)\in E_{p}$ for every $p\in M$, $v\in E_{p}$ and $t\in[0,T)$. For every $t\in[0,T)$, let $K(t)\subset E$ be a closed subbundle, invariant under parallel translation with respect to $D(t)$, convex in the fibers and such that the space time track
%	\begin{equation}
%		\mathcal{T}=\{(v,t)\in E\times\mathbb{R}:v\in K(t),\ t\in[0,T)\}
%	\end{equation}
%	is closed in $E\times[0,T)$. Suppose that for every $t_{0}\in[0,T)$, $K(t_{0})$ is preserved by the associated ODE
%	\begin{equation}\label{ODEparabequat}
%		\frac{\partial}{\partial t}Q=F(Q,t)
%	\end{equation}
%	i.e., any solution $Q(t)$ of the ODE (\ref{ODEparabequat}) that starts in $K(t_{0})_{p}$ remains in $K(t)_{p}$, as long as it exists. If $u(0)$ is contained in $K(0)$, then $u(p,t)\in K(t)_{P}$, for every $p\in M$, $t\in[0,T)$.	
%\end{theorem}
Both theorems provide sufficient conditions to find subsets of $End_{SA}(\wedge^{2}T_{p}M)$, the bundle of self-adjoint endomorphisms of $\wedge^{2}V$, which are preserved by PDE $(\ref{PDE})$. We refer the reader to Section 4 of \cite{catino1} for the precise statements and further details.

One of the conditions required by the maximum principles for a subset of $End_{SA}(\wedge^{2}T_{p}M)$ to be preserved by PDE $(\ref{PDE})$, is the invariance of this set by the ODE associated to equation (\ref{PDE}), obtained by dropping the elliptic operator, which reads
\begin{equation}\label{associatedODE}
	\frac{\partial}{\partial t}Q=2Q^{2}+2Q^{\#}-4\rho tr_{g_{0}}(Q)Q.
\end{equation}
In the particular case of $n=3$, a family of operators $Q_{p}(t)\in End_{SA}(\wedge^{2}T_{p}M)$ solves (\ref{associatedODE}) if and only if its eigenvalues $\lambda,\ \mu$ and $\nu$ satisfy the system 

\begin{equation}\label{system}
	\left\{
	\begin{array}[pos]{ll}
		\lambda'=2\lambda^2+2\mu\nu-4\rho\lambda(\lambda+\mu+\nu)\\\noalign{\smallskip}
		\mu'=2\mu^{2}+2\lambda\nu-4\rho\mu(\lambda+\mu+\nu)\\\noalign{\smallskip}
		\nu'=2\nu^{2}+2\lambda\mu-4\rho\nu(\lambda+\mu+\nu).
	\end{array}
	\right.
\end{equation}
We observe that the inequalities $\lambda\geq\mu\geq\nu$ are preserved by (\ref{system}). These inequalities are going to be assume from now on. For more details, see \cite{catino1,ham1}.

\section{Proofs}

Now we use the Vector Maximum Principle \cite[Theorem 4.8]{catino1} to prove Theorem \ref{HYestimate2}.

\begin{proof}[{\bf Proof of Theorem \ref{HYestimate2}}]
Consider the set $X_{p}^{\rho}$ of all $\mathcal{O}_{p}\in End_{SA}(\wedge^{2}T_{p}M)$ satisfying
\begin{equation}\label{estimate2'}
	\left\{
	\begin{array}[pos]{ll}
		\lambda+\mu+\nu\geq-\frac{e^{1-4\rho}}{2(1-2\rho)},\\\noalign{\smallskip}
		\mu+\nu\geq-f^{-1}(\lambda+\mu+\nu).
	\end{array}
	\right.
\end{equation}
Notice that $(\ref{estimate2'})$ is equivalent to $(\ref{HYestimate2})$. In order to prove Theorem \ref{HYestimate2}, we need to prove that $X_{p}^{\rho}$ is preserved by PDE $(\ref{PDE})$. This will be done by using the Vector Maximum Principle \cite[Theorem 4.8]{catino1}. It is easy to see that $X_{p}^{\rho}$ is closed and invariant by parallel translations. To see that it is convex, observe that $f^{-1}$ is concave, $\lambda+\mu+\nu$ is linear, and that $\mu+\nu$ is the least eigenvalue of the Ricci tensor, and hence, concave. 

Now we prove that $X_{p}^{\rho}$ is invariant by ODE $(\ref{system})$. This task is accomplished by proving that $(\ref{estimate2'})$ is preserved at the boundary of $X_{p}^{\rho}$. The first inequality of $(\ref{estimate2'})$ is clearly preserved, once
\begin{eqnarray}\label{import_inequ}
	(tr(Q))'\geq\frac{4}{3}(1-3\rho)tr(Q)^{2}\geq0,
\end{eqnarray}
on $[0,T)$, which is a general inequality obtained by adding the three equations of (\ref{system}). Now we deal with the second inequality of $(\ref{estimate2'})$. It is sufficient to assume that equality holds. In this case, using $(\ref{deff})$ we get
\begin{align}\label{sombonpon}
	\mu+\nu=-f^{-1}(\lambda+\mu+\nu)\leq-e^{1-4\rho}<0.
\end{align}
Now observe that at the points that satisfy the inequality above, the second inequality of $(\ref{estimate2'})$ is equivalent to $\lambda+\mu+\nu\geq f(-\mu-\nu)$, and by $(\ref{deff})$, equivalent to the inequality
\begin{equation}\label{estimate2''}
	\Lambda=-\frac{\lambda}{\mu+\nu}-\frac{1}{2(1-2\rho)}\ln(-\mu-\nu)\geq0.
\end{equation}
We will show that at the points satisfying $(\ref{sombonpon})$, $\Lambda'\geq0$. Using $(\ref{system})$, we conclude that $\Lambda'=2(\mu+\nu)^{-2}J$, where
\begin{align}\label{computed}
	\begin{split}
		J=&\lambda(\mu^2+\nu^2)-(\mu+\nu)\mu\nu-\frac{1}{2(1-2\rho)}(\mu+\nu)(\mu^2+\nu^2)\\
		  &-\frac{1}{2(1-2\rho)}\lambda(\mu+\nu)^2+\frac{\rho}{1-2\rho}(\mu+\nu)^2(\lambda+\mu+\nu).
	\end{split}
\end{align}
We need to show that $J\geq0$. We divide the proof into two cases, according to the sign of $\lambda+\mu+\nu$. From $(\ref{sombonpon})$, we already have $\mu+\nu<0$ and $\nu<0$.

Suppose that $\lambda+\mu+\nu<0$. In this case we can estimate $(\ref{computed})$ in the following way
\begin{align*}
		J=&\lambda(\mu^2+\nu^2)-(\mu+\nu)\mu\nu-\frac{1}{1-2\rho}\lambda\mu\nu\\
		&\frac{\rho}{1-2\rho}(\mu+\nu)^2(\lambda+\mu+\nu)-\frac{1}{2(1-2\rho)}(\lambda+\mu+\nu)(\mu^2+\nu^2)\\
		\geq&\lambda(\mu^2+\nu^2)-(\mu+\nu)\mu\nu-\frac{1}{1-2\rho}\lambda\mu\nu.
\end{align*}
If either $\lambda\leq0$ or $\mu\geq0$, then the last term in the inequality above is nonnegative. Consequently, $J\geq\mu^2(\lambda-\nu)+\nu^2(\lambda-\mu)\geq0$. If $\mu\leq0\leq\lambda$, since $\rho<0$, we have $J\geq\lambda\nu^2-\lambda\mu\nu=-\lambda\nu(\mu-\nu)\geq0$. Thus, $J\geq0$ in this case.

Now suppose that $\lambda+\mu+\nu\geq0$. Observe that $\lambda\geq0$ in this case. From $(\ref{computed})$ we have
\begin{align*}
		J-\frac{\rho}{1-2\rho}(\mu+\nu)^3=&\lambda(\mu^2+\nu^2)-(\mu+\nu)\mu\nu-\frac{1}{2(1-2\rho)}(\mu+\nu)(\mu^2+\nu^2)-\frac{1}{2}\lambda(\mu+\nu)^2\\
		\geq&\lambda(\mu^2+\nu^2)-(\mu+\nu)\mu\nu-\frac{1}{2}\lambda(\mu^2+\nu^2)-\lambda\mu\nu\\
		=&\frac{1}{2}\lambda(\mu^2+\nu^2)-(\mu+\nu)\mu\nu-\lambda\mu\nu
\end{align*}
Now suppose that $\mu\leq0$. Then
\begin{align*}
		J-\frac{\rho}{1-2\rho}(\mu+\nu)^3&\geq\frac{1}{2}\lambda(\mu^2-2\mu\nu+\nu^2)-(\mu+\nu)\mu\nu\\
		&=\frac{1}{2}\lambda(\mu-\nu)^2-(\mu+\nu)\mu\nu\geq0.
\end{align*}
On the other hand, if $\mu\geq0$, then 
\begin{align*}
	J-\frac{\rho}{1-2\rho}(\mu+\nu)^3\geq\frac{1}{2}\lambda(\mu^2+\nu^2)-(\lambda+\mu+\nu)\mu\nu\geq0.
\end{align*}
In any case, if $\lambda+\mu+\nu\geq0$, then
\begin{align}\label{uselater}
	J\geq\frac{\rho}{1-2\rho}(\mu+\nu)^3\geq0.
\end{align}
This finishes the proof.
\end{proof}

In order to proceed with the proofs, we consider time dependent subsets of $End_{SA}(\wedge^{2}V)$ for which we intend to use the time dependent version of the Vector Maximum Principle.

Given $\mathcal{O}_{p}\in End_{SA}(\wedge^{2}T_{p}M)$, denote its ordered eigenvalues by $\lambda\geq\mu\geq\nu$. Consider real numbers $\eta,\rho$ and $\theta>0$. For each $t\in[0,T)$, define the properties $(P_{1})_{t}$, $(P_{2})_{t}$ and $(P_{3})_{t}$ as: if $1+\eta\rho>0$, put
\begin{align*}
	\hspace{-2.1cm}(P_{1})_{t} &\ \ \lambda+\mu+\nu\geq-\displaystyle\frac{3}{1+2(1+\eta\rho)t},\\\noalign{\smallskip}
	\hspace{-2.1cm}(P_{2})_{t} &\ \ \text{If}\ \nu\leq-\displaystyle\frac{1}{1+2(1+\eta\rho)t},\ \text{then} \\
	\hspace{-2.1cm}&\lambda+\mu+\nu\geq-\theta\nu(\log(-\nu)+\log(1+2(1+\eta\rho)t)-3\theta^{-1}),
\end{align*}
and if $\rho<0$, put
\begin{align*}
	(P_{3})_{t} &\ \ \text{If}\ \mu+\nu\leq\displaystyle-\frac{1}{1-4\rho t},\ \text{then} \\\noalign{\smallskip}
	 &\ \ \lambda+\mu+\nu\geq-\frac{1}{2(1-2\rho)}(\mu+\nu)(\log(-\mu-\nu)+\log(1-4\rho t)-2(1-2\rho)).
\end{align*}
Finally, for each $t\in[0,T)$, we define the time dependent sets
\begin{align*}
	&K^{\eta,\rho}_{p}(t)=\left\{ \mathcal{O}_{p}\in End_{SA}(\wedge^{2}T_{p}M):\mathcal{O}_{p}\ \text{satisfies}\ (P_{1})_{t} \ \text{and}\ (P_{2})_{t}\right\},\\
		&Y^{\eta,\rho}_{p}(t)=\left\{ \mathcal{O}_{p}\in K^{\eta,\rho}_{p}(t):\mathcal{O}_{p}\ \text{satisfies}\ \mu+\nu\geq0 \right\},\\
	&W^{\rho}_{p}(t)=\left\{ \mathcal{O}_{p}\in End_{SA}(\wedge^{2}T_{p}M):\mathcal{O}_{p}\ \text{satisfies}\ \lambda+\mu+\nu\geq0 \ \text{and}\ (P_{3})_{t}\right\}.
\end{align*}

We note that when $\rho=0$, then $K^{\eta,0}_{p}(t)$ agrees with the set used by Hamilton in \cite{ham1} to prove the classical Hamilton-Ivey estimate. When $\rho\in[0,1/6)$ and $\eta=-6$, then $K^{\eta,\rho}_{p}(t)$ agrees with the set used by Catino \textit{et al.} in \cite{catino1}, placing $\theta=1$.

In order to use the time dependent version of Vector Maximum Principle, \cite[Theorem 4.9]{catino1}, we need to check the following properties for a set $K(t)\subset End_{SA}(\wedge^{2}TM)$:
\begin{enumerate}
	\item[(1)] Convexity;
	\item[(2)] Invariance under parallel translations;
	\item[(3)] Closedness of the set, $\{(v,t)\in E\times\mathbb{R}:v\in K(t),\ t\in[0,T)\}$, the track of $K(t)$;
	\item[(4)] Invariance of $K_{p}(t)$ under the system of ODE's (\ref{system}).
\end{enumerate}

When $\rho\in[0,1/6)$ and $\eta=-6$, $K^{\eta,\rho}_{p}(t)$ was proven to satisfy $(1)$-$(4)$. This proof can be found in \cite{ham1,catino1}. The corresponding proofs of $(1)$-$(3)$ can be adapted with no significant changes for any $\eta$ and $\rho$ so that $1+\eta\rho>0$. Similar arguments show that the sets $K^{\eta,\rho}_{p}(t)$, $Y^{\eta,\rho}_{p}(t)$ and $W^{\rho}_{p}(t)$ satisfy $(1)$-$(3)$. This is the content of the next result.
\begin{proposition}\label{otherissues}
	The sets $K^{\eta,\rho}_{p}(t)$, $Y^{\eta,\rho}_{p}(t)$ and $W^{\rho}_{p}(t)$ are convex, invariant under parallel translation and have closed space time track.
\end{proposition}

Thus, the only property we have left to verify is the invariance of these sets by $(\ref{system})$. The following proposition is motivated from works of Hamilton and Ivey.

\begin{proposition}\label{proptobe2}
	Let $\rho<0$ and consider a solution $\lambda,\ \mu\ \text{and}\ \nu$ of $(\ref{system})$, corresponding to $\rho$ and defined in $[0,T)$. Consider the function $\eta:[0,T)\rightarrow\mathbb{R}$ defined by
	\begin{equation}\label{auxgeneral2}
		\begin{array}[pos]{cc}
			\displaystyle\Lambda=-\frac{\lambda}{\mu+\nu}-\frac{1}{2(1-2\rho)}\ln(-\mu-\nu).
		\end{array}
	\end{equation}
	If $\Lambda'\geq\frac{2\rho}{1-2\rho}(\mu+\nu)$ in every subset $[T_{1},T_{2}]\subset[0,T)$ for which
	\begin{equation}\label{ineq.2}
		\lambda+\mu+\nu\geq0\ \text{and}\ \mu+\nu\leq\displaystyle-\frac{1}{1-4\rho t},\ t\in[T_{1},T_{2}],
	\end{equation}
	then $W^{\rho}_{p}(t)$ is preserved by (\ref{system}).
\end{proposition}
\begin{proof}
	Let $Q_{p}(t)=Q(t)\in End_{SA}(\wedge^{2}T_{p}M)$, $t\in[0,T)$, be a solution to the ODE (\ref{associatedODE}) with $Q(0)\in W^{\rho}_{p}(0)$ and fix $t_{0}\in(0,T)$. We will show that $Q(t_{0})\in W^{\rho}_{p}(t_{0})$.
	
	It follows from (\ref{import_inequ}) that the first inequality of $(\ref{ineq.2})$ is preserved. Now we consider the second inequality.	If the equality holds, i.e., $-(\mu(t_{0})+\nu(t_{0}))(1-4\rho t_{0})=1$, then we get at $t_{0}$
	\begin{align*}\label{torefer}
		\begin{split}
			\lambda+\mu+\nu&\geq\mu+\nu=-\frac{1}{2(1-2\rho)}(\mu+\nu)(-2(1-2\rho))\\
			&=-\frac{1}{2(1-2\rho)}(\mu+\nu)(\log(-\mu-\nu)+\log(1-4\rho t_{0})-2(1-2\rho)),
		\end{split}
	\end{align*}
	where we have used $\lambda\geq0$. If, on the other hand, the inequality $(\ref{ineq.2})$ is strict, i.e., if
	\begin{equation}\label{inequality}
		\mu+\nu<\displaystyle-\frac{1}{1-4\rho t_{0}},
	\end{equation}
	then consider the smallest number $\tilde{t}\in[0,t_{0})$ such that inequality $(\ref{inequality})$ remains true for all $t\in(\tilde{t},t_{0}]$. Consequently, $(\ref{ineq.2})$ is true in $[\tilde{t},t_{0}]$, which from our hypothesis implies that
	\begin{align}\label{nondecreasing}
		\left(\Lambda-\frac{1}{2(1-2\rho)}\ln(1-4\rho t)\right)'\geq0,
	\end{align}
	$\forall t\in[\tilde{t},t_{0}]$. Now we consider two situations. If $\tilde{t}=0$, it follows from $Q(0)\in W^{\rho}_{p}(0)$ that
	\begin{equation}\label{initialc1}
		\Lambda(0)=-\frac{\lambda(0)}{\mu(0)+\nu(0)}-\frac{1}{2(1-2\rho)}\ln(-\mu(0)-\nu(0))\geq0.
	\end{equation}
	If $\tilde{t}>0$, It follows from the continuity of $\mu+\nu$ that $(1-4\rho \tilde{t})(\mu(\tilde{t})+\nu(\tilde{t}))=-1$,	and then 
	\begin{align}\label{initialc2}
		\Lambda(\tilde{t})-\frac{1}{2(1-2\rho)}\ln(1-4\rho \tilde{t})=-\frac{\lambda(\tilde{t})}{\mu(\tilde{t})+\nu(\tilde{t})}\geq0.
	\end{align}
	Now, from $(\ref{nondecreasing})$, $(\ref{initialc1})$ and $(\ref{initialc2})$ we have
	\begin{align}
		\Lambda(t_{0})\geq\frac{1}{2(1-2\rho)}\ln(1-4\rho t_{0})
	\end{align}
	which is equivalent to
	\begin{align*}
		\begin{split}
			\lambda+\mu+\nu\geq-\frac{1}{2(1-2\rho)}(\mu+\nu)(\log(-\mu-\nu)+\log(1-4\rho t_{0})-2(1-2\rho)),
		\end{split}
	\end{align*}
	at $t_{0}$. This finally proves that $Q(t_{0})\in W_{p}^{\rho}(t_{0})$.	
\end{proof}

Now we prove Theorem \ref{HYestimatewithR}.
\begin{proof}[{\bf Proof of Theorem \ref{HYestimatewithR}}]
	According to Proposition \ref{otherissues} and Proposition \ref{proptobe2}, we only need to proof the estimate $\Lambda'\geq\frac{2\rho}{1-2\rho}(\mu+\nu)$, at every set $[T_{1},T_{2}]\subset[0,T)$ where $(\ref{ineq.2})$ happens. Notice that this is exactly estimate $(\ref{uselater})$, proved under the conditions $\lambda+\mu+\nu\geq0$ and $\mu+\nu<0$. This shows the estimate, which finishes the proof.
\end{proof}

Now we state the following proposition, whose proof is analogous to the proof of Proposition \ref{proptobe2} and can be found in \cite[Theorem 4.15]{catino1}.
\begin{proposition}\label{proptobe}
	Consider a solution $\lambda,\ \mu\ \text{and}\ \nu$ of $(\ref{system})$, corresponding to $\rho\in(-\infty,1/4)$, defined in $[0,T)$. Let also and fix $\theta>0$ and $\eta$ so that $1+\eta\rho>0$. Consider the function 
	\begin{equation}\label{auxgeneral}
		\begin{array}[pos]{cc}
			\displaystyle \xi(t)=\frac{\lambda+\mu+\nu}{-\nu}-\theta\ln(-\nu)-\theta\ln(1+2(1+\eta\rho)t).
		\end{array}
	\end{equation}
	\begin{enumerate}
		\item Suppose that $\eta=-4$ and $\rho\in[0,1/4)$. If $\xi(t)$ is nondecreasing in every subset $[T_{1},T_{2}]\subset[0,T)$ for which
		\begin{equation}\label{ineq.}
			\nu\leq-\frac{1}{1+2(1-4\rho)t},\ t\in[T_{1},T_{2}],
		\end{equation}
		then $K^{\eta,\rho}_{p}(t)$ is preserved by (\ref{system}).
		\item Suppose that $\eta>0$ and $\rho\in(-1/\eta,0]$. If $\xi(t)$ is nondecreasing in every subset $[T_{1},T_{2}]\subset[0,T)$ for which
		\begin{equation}\label{ineq..}
			\mu+\nu\geq0\ \ \text{and}\ \ \nu\leq-\frac{1}{1+2(1+\eta\rho)t},\ t\in[T_{1},T_{2}],
		\end{equation}
		then $Y^{\eta,\rho}_{p}(t)$ is preserved by (\ref{system}).
	\end{enumerate}
	
\end{proposition}

We finally have:

\begin{proof}[{\bf Proof of Theorem \ref{HYestimateRicci>0}}]
	As in the proof of Theorem \ref{HYestimate}, we only need to show that $\xi(t)$ is nondecreasing. For this, we take $\theta=-1/2\rho$. First recall that, by hypothesis, $\rho<0$, $1+\eta\rho>0$, $\nu<0$ and $\mu+\nu\geq0$. These last two inequalities imply that $\lambda\geq\mu\geq-\nu>0$. On the other hand, by $(\ref{system})$ and $(\ref{ineq..})$ we get
	\begin{align*}
		\nu^2\displaystyle \xi'(t)\geq&(\lambda'+\mu')(-\nu)+(\lambda+\mu)\nu'-\theta\nu\nu'+2\theta(1+\eta\rho)\nu^3\\
		\geq&-2\nu(\lambda^2+\mu^2)+2\lambda\mu(\lambda+\mu)-2\theta\nu(\nu^2+\lambda\mu-2\rho\nu(\lambda+\mu+\nu))+2\theta(1+\eta\rho)\nu^3\\
		\geq&2\lambda\mu(\lambda+\mu)+4\theta\rho\nu^2(\lambda+\mu)-2\theta(1-2\rho)\nu^3+2\theta(1+\eta\rho)\nu^3\\
		=&2(\lambda+\mu)(\lambda\mu+2\theta\rho\nu^2)+2\theta(-1+2\rho+1+\eta\rho)\nu^3\\
		\geq&2(\lambda+\mu)\nu^2(1+2\theta\rho)+2\theta(2+\eta)\rho\nu^3\\
		\geq&0,
	\end{align*}
	where we have used that  $2\rho\theta=-1$. This proves that $\xi(t)$ is nondecreasing in $[T_{1},T_{2}]$, which finishes the proof.
\end{proof}

We finish the paper with the proof of the Hamilton-Ivey estimate for $\rho\in[0,1/4)$, without additional assumptions.

\begin{theorem}\label{HYestimate}
	Let $M^{3}$ be a compact three manifold, $\rho\in[0,1/4)$ and $g_{0}$ a Riemannian metric on $M$ satisfying the normalizing assumption $\displaystyle\min_{p\in M}\nu_{0}(p)\geq-1$, where $\nu_{0}$ is the smallest sectional curvature of $g_{0}$. If $g(t)$, $t\in[0,T)$, is the solution of $(\ref{defRB})$ with $g(0)=g_{0}$, corresponding to $\rho$, then the scalar curvature $R(t)$ of $g(t)$ satisfies
	\begin{eqnarray}\label{estimate1}
		R\geq2|\nu|(\log|\nu|+\log(1+2(1-4\rho)t)-3),
	\end{eqnarray}
	at any point $(p,t)$ where the smallest sectional curvature  $\nu(p,t)$ of $g_{p}(t)$ is negative.
\end{theorem}
\begin{proof}
	By Proposition \ref{otherissues}, the only property left to show is the invariance of $K$ under $(\ref{system})$. For this, we take $\theta=1$. By Proposition \ref{proptobe}, we only need to show that $\xi(t)$ is nondecreasing. First of all, note that (\ref{system}) and $(\ref{ineq.})$ give
	\begin{align*}
		\xi'(t)&\geq\displaystyle\left(\frac{\lambda+\mu+\nu}{-\nu}\right)'-(\ln(-\nu))'+2(1-4\rho)\nu\nonumber\\
		%			&=\nu^{-2}[-2\nu^{3}-2\nu(\lambda^{2}+\mu^{2})+2\mu\lambda(\mu+\lambda)-2\nu\mu\lambda+4\rho\nu^{2}(\lambda+\mu+\nu)+2(1-4\rho)\nu^{3}]\nonumber\\
		&=\nu^{-2}[-2\nu(\lambda^{2}+\mu^{2})+2\mu\lambda(\mu+\lambda)-2\nu\mu\lambda+4\rho\nu^{2}(\lambda+\mu)-4\rho\nu^{3}]\nonumber\\
		&=\nu^{-2}I.
	\end{align*}
	In what follows we will show that $I>0$, which together with the equalities above, implies that $\xi(t)$ is nondecreasing. To reach our goal we are going to consider the following cases.
	
	Assume that $\lambda\geq0$. If $\mu\leq0$ then
	\begin{equation*}
		\begin{array}[pos]{lll}
			I&=&2(\mu-\nu)(\lambda^{2}+\lambda\mu+\mu^{2})-2\mu(\lambda^{2}+\lambda\mu+\mu^{2})+2\mu\lambda(\mu+\lambda)+4\rho\nu^{2}\lambda+4\rho\nu^{2}(\mu-\nu)\\\noalign{\smallskip}
			&=&2(\mu-\nu)(\lambda^{2}+\lambda\mu+\mu^{2})-2\mu^{3}+4\rho\nu^{2}\lambda+4\rho\nu^{2}(\mu-\nu)\\\noalign{\smallskip}
			&\geq&0.
		\end{array}
	\end{equation*}
	On the other hand, if $\mu>0$ we get
	\begin{equation*}
		\begin{array}[pos]{lll}
			I&=&-2\nu(\lambda^{2}+\mu^{2})+2\mu\lambda(\mu+\lambda)-2\nu\mu\lambda+4\rho\nu^{2}\lambda\\\noalign{\smallskip}
			&\geq&0.	
		\end{array}
	\end{equation*}
	
	Now assume that $\lambda<0$. If $2\mu\leq\nu$, then
	\begin{equation*}
		\begin{array}[pos]{lll}
			I&=&-2\nu(\lambda^{2}+\mu^{2})+2\mu\lambda(\mu+\lambda)-2\nu\mu\lambda+8\rho\nu\mu\lambda-8\rho\nu\mu\lambda+4\rho\nu^{2}\lambda\\\noalign{\smallskip}
			&=&2\lambda^{2}(\mu-\nu)+2\mu^{2}(\lambda-\nu)+2(4\rho-1)\nu\mu\lambda+4\rho\nu\lambda(\nu-2\mu)\\\noalign{\smallskip}
			&\geq&0.
		\end{array}
	\end{equation*}
	If, on the other hand $2\mu>\nu$, then
	\begin{equation*}
		\begin{array}[pos]{lll}
			I&=&-2\nu(\lambda^{2}+\mu^{2})+2\mu\lambda(\mu+\lambda)-2\nu\mu\lambda+4\rho\nu^{2}(\lambda+\mu)-4\rho\nu^{3}\\\noalign{\smallskip}
			&=&2\lambda^{2}(\mu-\nu)+2\mu^{2}(\lambda-\nu)-2\nu\mu\lambda+4\rho\nu^{2}(\lambda+\mu-\nu)\\\noalign{\smallskip}
			&\geq&2\lambda^{2}(\mu-\nu)+2\mu^{2}(\lambda-\nu)-2\nu\mu\lambda+4\rho\nu^{2}(2\mu-\nu)\\\noalign{\smallskip}
			&\geq&0.
		\end{array}
	\end{equation*}
	
	In all cases we have $I\geq0$, which finishes the proof.
\end{proof}

\address{Universidade Federal do Pará - UFPA\\
Faculdade de Matem\'{a}tica, 66075-110 Be\-l\'{e}m, Par\'{a} - PA, Brazil.\\
\email{valterborges@ufpa.br}}

\end{document}